\documentclass[a4paper,leqno,11pt]{amsart}

\usepackage{amssymb}
\usepackage{amsthm}
\usepackage{tikz-cd}
\usepackage{mathtools}
\usepackage{framed}
\usepackage[mathscr]{eucal}
\usepackage{fullpage}
\usepackage{hyperref}
\usepackage{enumitem}
\usepackage{wasysym}
\usepackage{xcolor}
\usepackage{scalerel}
\usepackage{comment}
\usepackage[all]{xy}
\hypersetup{colorlinks=true,citecolor=blue,urlcolor =black,linkbordercolor={1 0 0}}
\mathtoolsset{centercolon}
\title{Twisted Arinkin transforms and derived categories of moduli spaces on Kuznetsov components}
\author{Moritz Hartlieb}
\address{Mathematisches Institut, Universität Bonn, Endenicher Allee 60, 53115
Bonn, Germany}
\email{hartlieb@math.uni-bonn.de}

\author{Saket Shah} 
\address{Department of Mathematics, University of Michigan, Ann Arbor, MI 48109, USA}
\email{sakets@umich.edu}
\date{\today}

\newcommand{\cO}{\mathcal{O}}
\newcommand{\cA}{\mathcal{A}}

\newcommand{\cC}{\mathcal{C}}

\newcommand{\cK}{\mathcal{K}}
\newcommand{\cL}{\mathcal{L}}

\newcommand{\cP}{\mathcal{P}}
\newcommand{\cQ}{\mathcal{Q}}

\newcommand{\cS}{\mathcal{S}}

\newcommand{\bbC}{\mathbb{C}}

\newcommand{\bbG}{\mathbb{G}}

\newcommand{\bbQ}{\mathbb{Q}}

\newcommand{\bbZ}{\mathbb{Z}}

\newcommand{\congpf}{\xymatrix@L=0.6ex@1@=15pt{\ar[r]^-\sim&}}

\newcommand{\Hom}{\operatorname{Hom}}
\newcommand{\Pic}{\operatorname{Pic}}
\newcommand{\Br}{\operatorname{Br}}
\newcommand{\NS}{\operatorname{NS}}

\newcommand{\SBr}{\operatorname{SBr}}
\newcommand{\Db}{\operatorname{D}^b}
\newcommand{\Ku}{\cK u}
\newcommand{\Stab}{\operatorname{Stab}}

\DeclareMathOperator{\Sh}{Sh}
\DeclareMathOperator{\Hilb}{Hilb}

\usepackage[T2A]{fontenc}
\DeclareSymbolFont{cyrillic}{T2A}{cmr}{m}{n}
\DeclareMathSymbol{\Sha}{\mathalpha}{cyrillic}{216}

\newtheorem{theorem}{Theorem}[section]
\newtheorem{lemma}[theorem]{Lemma}
\newtheorem{proposition}[theorem]{Proposition}
\newtheorem{conjecture}[theorem]{Conjecture}
\newtheorem{corollary}[theorem]{Corollary}

\theoremstyle{definition}
\newtheorem{definition}[theorem]{Definition}

\newtheorem{remark}[theorem]{Remark}

\DeclareMathOperator{\pr}{pr}

\newcommand{\transl}{\tau}

\usepackage[
  backend=biber,
  style=alphabetic,
  sorting=nty,
  maxcitenames=50,
  maxnames=50
]{biblatex}

\renewbibmacro{in:}{}
\AtEveryBibitem{%
  \clearfield{issn} 
  \clearfield{doi} 
  \clearfield{isbn}
  \ifentrytype{online}{}{
    \clearfield{url}
  }
}

\DeclareSourcemap{
  \maps[datatype=bibtex]{
    \map{
      \step[typesource=arxiv,       typetarget=misc]
      \step[fieldset=entrysubtype, fieldvalue=arxiv]
    }
  }
}

\DeclareFieldFormat[misc]{title}{%
  \iffieldequalstr{entrysubtype}{arxiv}
    {\mkbibquote{#1}}
    {\mkbibemph{#1}}}

\setlist[enumerate]{label={\rm{(\roman*)}}}

\numberwithin{equation}{section}

\begin{document}

\begin{abstract}
    In this note, we generalize results of Donagi and Pantev on twisted derived equivalences between elliptically fibered surfaces to higher dimensions. First, we establish a twisted derived equivalence between torsors under abelian schemes satisfying a certain compatibility condition. Then, relying on the work of Arinkin on compactified Jacobians, we extend the equivalence to twisted compactified Jacobians associated to curves on K3 surfaces. This positively answers a question stated by Mattei and Meinsma in \cite{matteimeinsma2025}. We then extend methods of Bottini and Huybrechts for Fano varieties of lines on cubic fourfolds to prove an analogue of their main theorem in \cite{huybrechtsbottini} for general moduli spaces of Bridgeland-stable objects on Kuznetsov components admitting rational Lagrangian fibrations. 
\end{abstract}
\maketitle
\tableofcontents

\section{Introduction}
Let $A$ be an abelian variety over $\mathbb C$. In \cite{mukaiduality}, Mukai has shown that the Poincar\'e sheaf induces a derived equivalence
$$D^b(A) \simeq D^b(\check{A})$$
between $A$ and its dual abelian variety $\check{A}$. More generally, if $A / B$ is an abelian scheme, then the relative Poincar\'e sheaf induces a derived equivalence between  $A$ and its dual abelian scheme $\check{A} / B.$ Suppose that $X$ is a torsor under $A$. By interpreting $\check{A} = \Pic^0_{A/B}\simeq\Pic^0_{X/B}$ as a moduli space of line bundles on $X$, one obtains a twisted Poincar\'e sheaf on $X \times \check{A}$, which induces an equivalence
$$D^b(X) \simeq D^b(\check{A}, \alpha),$$
where $\alpha \in \Br(\check{A})$ is the Brauer class obstructing $\check{A}$ from being a fine moduli space of line bundles on $X$.

One of the aims of this note is to study generalizations of the above to obtain equivalences between compactifications of torsors under abelian schemes, generalizing results on elliptic fibrations by Donagi and Pantev \cite{donagipantev}. In the setting of abelian schemes, we prove the following
\begin{theorem}[{Thm.\ \ref{thm:main_abelian_schemes}}]
    \label{thm:main_abelian_schemes_intro}
    Let $A$ be an abelian scheme over a quasi-projective non-singular variety $B$. Let $X$ be a torsor under $A$ and $Y$ a torsor under the dual abelian scheme $\check{A}$. Assume that $[X] \in H^1(B, \mathcal A)$ is $n$-torsion and $[Y] \in H^1(B, \check{\mathcal A})$ is $n$-divisible. Then, there are Brauer classes $\alpha \in \Br(X)$ and $\beta \in \Br(Y)$ and exact linear equivalences
    $$D^b(X, \alpha \cdot \pi_X^*\gamma) \simeq D^b(Y, \beta \cdot \pi_Y^* \gamma)$$
    for all $\gamma \in \Br(B).$
\end{theorem}
For a more general version of Theorem \ref{thm:main_abelian_schemes_intro}, allowing suitable singular fibers, see Theorem \ref{thm:main_abelian_schemes}. A discussion on the torsion (resp.\ divisibility) assumption may be found in \cite[Sec.\ 2.3]{donagipantev}.

\begin{remark}
    During the preparation of this note, we found that Theorem \ref{thm:main_abelian_schemes_intro} is very close to a result by Ben-Bassat \cite[Cor.\ 6.2]{benbassat}, which follows from a more general formalism for twisting derived equivalences. However, we believe it is worth to provide our alternative proof as we need not assume that the Brauer group $\Br(B)$ is trivial and our description of the Brauer classes is somewhat more explicit, cf.\ Remark \ref{rem:benbassat}.
\end{remark}

The remainder of this note is concerned with applications of Theorem \ref{thm:main_abelian_schemes_intro} to Lagrangian-fibered hyperkähler manifolds of K3$^{[n]}$-type. Let us recall the results of Donagi and Pantev on twisted derived equivalences between elliptic K3 surfaces:
Let $S \to \mathbb P^1$ be an elliptic K3 surface with a section (and at worst nodal fibers). Then, we have $\Br(S) \simeq\Sha(S)$, where the later group is the Tate--Shafarevich group parametrizing compactifications of torsors under the generic fiber of $S \to \mathbb P^1$, see \cite[Sec.\ 11.5]{lecK3}. In particular, for every $\alpha \in \Br(S)$, there is an elliptic K3 surface $S_{\alpha} \to \mathbb P^1$, whose generic fiber is a torsor under the generic fiber of $S / \mathbb P^1$. Hodge theory provides us with a restriction map $o_{\alpha} \colon \Br(S) \to \Br(S_{\alpha})$, the kernel of which is generated by $\alpha$. In the seminal article \cite{donagipantev}, Donagi and Pantev proved the following generalization of Mukai's equivalence to elliptic K3 surfaces.
\begin{theorem}[{\cite[Thm.\ A]{donagipantev}}]\label{thm:donagi_pantev}
    Let $S$ be an elliptic K3 surface with a section and at worst nodal fibers. Then, for any pair of Brauer classes $\alpha, \beta \in \Br(S)$, there is an exact linear equivalence
    $$D^b(S_{\alpha}, o_{\alpha}(\beta)^{-1}) \simeq D^b(S_{\beta}, o_{\beta}(\alpha)).$$
\end{theorem}
We generalize the above theorem to higher dimensions: Let $(S, L)$ be a polarized K3 surface and assume that all members of the linear system $|L|$ are integral curves. Then, in \cite{huybrechtsmattei}, Huybrechts and Mattei introduced the 'primitive' special Brauer group $\SBr(S, L)$, which parametrizes twists of the compactified Jacobian
$\overline{\Pic}^0(\mathcal C / |L|)  \to |L|$ of the universal curve $\mathcal C / |L|$. In particular, for $\alpha \in \SBr(S, L)$, there is a Lagrangian fibered Hyperkähler manifold $\overline{\Pic}^0_{\alpha} \to |L|$, the generic fiber of which is a torsor over the generic fiber of $\overline{\Pic}^0 \to |L|$, see Section \ref{sec:background_twistedcompact}. For any $\alpha \in \SBr(S, L)$, we construct a group homomorphism
$$o_{\alpha} \colon \SBr(S, L) \to \Br(\overline{\Pic}^0_{\alpha})$$
and prove the following generalization of Theorem \ref{thm:donagi_pantev}:
\begin{theorem}[{Thm.\ \ref{thm:main_hk}}]\label{thm:main_hk_intro}
    Let $\alpha, \beta \in \SBr(S,L)$ be special Brauer classes. Then, there is a twisted derived equivalence
    $$D^b(\overline{\Pic}^0_{\alpha}(\mathcal C / |L|), o_{\alpha}(\beta)^{-1}) \simeq D^b(\overline\Pic_{\beta}^0(\mathcal C / |L|), o_{\beta}(\alpha)).$$
\end{theorem}
Moreover, we identify certain instances of $o_{\alpha}(\beta)$ with Brauer classes arising from the interpretation of $\overline\Pic^0_{\alpha}$ as a moduli space of twisted sheaves on $S$, see Section \ref{sec:identifying_brauer}.
In particular, Theorem \ref{thm:main_hk_intro} generalizes \cite[Prop.\ 2.1]{adm}, see also \cite[Cor.\ 5.9]{matteimeinsma2025}, and \cite[Thm.\ 3.3]{bottini}, and provides a positive answer to the question posed after \cite[Cor.\ 5.9]{matteimeinsma2025}.

From Theorem \ref{thm:main_hk_intro}, we deduce the following description of twisted derived categories of Lagrangian fibered hyperkähler manifolds of K3$^{[n]}$-type:
\begin{corollary}[{Cor.\ \ref{cor:generallangrangianfibered}}]\label{cor:generallang}
    Let $f \colon X \to \mathbb P^n$ be a Lagrangian fibered hyperkähler manifold of K3$^{[n]}$-type and assume that $X$ has Picard rank $2$ and $f^* \mathcal O_{\mathbb P^n}(1)$ is indivisible in $H^2(X, \mathbb Z)$. Then, there is a twisted K3 surface $(S, \alpha)$, a Brauer class $\theta \in \Br(X)$ and a linear exact equivalence
    $$D^b(X, \theta^n) \simeq (D^b(S, \alpha))^{[n]}.$$
\end{corollary}
The proof of Corollary \ref{cor:generallang} relies on \cite[Thm.\ 1.2]{huybrechtsmattei}, see also \cite{markmanlagrangian}, which asserts that there is a complete linear system $\mathcal C \to |L|$ on $S$ such that $X$ is birational to $\Pic^0_{\alpha}(\mathcal C/|L|_{\mathrm{sm}})$. The assumption on the divisibility of $f^* \mathcal O_{\mathbb P^n}(1)$ is needed to ensure the integrality of all members of $|L|$. As the square of the divisibility divides $n-1$, see \cite[Lem.\ 2.5]{markmanlagrangian}, this is automatic if $n-1$ is square-free. Since the integrality of all members of $|L|$ is an open condition, the assumption on the Picard rank can be weakened to a genericity assumption, see Remark \ref{rem:genericity}.

Note that Corollary \ref{cor:generallang} can be viewed as evidence for the following conjecture by Zhang:
\begin{conjecture}[{Cf.\ \cite[Spec.\ 1.1]{zhang}}]\label{conj:zhang}
Let $X$ be a hyperkähler variety of K3$^{[n]}$-type with $n \geq 2$. Then, there exists a K3 category $\mathcal A_X$ such that $X$ is a moduli space of stable objects on $\mathcal A_X$, and a canonical Brauer class $\theta_X \in \Br(X)$ such that there is an exact linear equivalence
$$D^b(X, \theta_X^{\pm n}) \simeq \mathcal A_X^{[n]}.$$
\end{conjecture}
A lattice-theoretic description of the class $\theta_X$ is given in \cite[Sec.\ 2.2]{zhang}.

For more concrete applications, we focus on hyperkähler manifolds of K3$^{[n]}$-type that arise as moduli spaces of Bridgeland-stable objects in Kuznetsov components of cubic fourfolds. Let $Y \subset \mathbb P^5$ be a smooth cubic fourfold. Then, Li, Pertusi and Zhao have shown in \cite{lipertusizhao} that the Fano variety $F(Y)$ of lines on $Y$ is a moduli space of Bridgeland-stable objects on the Kuznetsov component $\mathcal A_Y \coloneqq \langle \mathcal O_Y, \mathcal O_Y(1), \mathcal O_Y(2) \rangle^\perp \subset D^b(Y)$
of $Y$. The following recent result by Kemboi and Segal may thus be interpreted as an instance of Conjecture \ref{conj:zhang}.
\begin{theorem}[{\cite{kemboisegal}}]\label{thm:kemboisegal}
    Let $Y \subset \mathbb P^5$ be a smooth cubic fourfold and let $F(Y)$ denote the Fano variety of lines on $Y$. Then, there is a linear exact equivalence
    $$D^b(F(Y)) \simeq \mathcal A_Y^{[2]},$$
    where $\mathcal A_Y \subset D^b(Y)$ denotes the Kuznetsov component of $Y$.
\end{theorem}

In \cite{huybrechtsbottini}, Bottini and Huybrechts gave a proof of Theorem \ref{thm:kemboisegal} in the special case where $F(X)$ admits a Lagrangian fibration. With Theorem \ref{thm:main_hk_intro} at hand, we refine their strategy in order to provide strong evidence for Conjecture \ref{conj:zhang}.
\begin{theorem}[{Thm.\ \ref{thm:main_kuznetsov_component}}]\label{thm:main_kuznetsov_component_intro}
    Let $Y$ be a smooth cubic fourfold. Let $v \in \widetilde{H}(\cA_Y, \mathbb Z)$ be a non-zero primitive vector and let $\sigma$ be a $v$-generic stability condition in $\Stab^{\dagger}(\mathcal A_Y)$. Furthermore, assume that the moduli space
    $M_{\sigma}(\mathcal A_Y, v)$
    is of Picard rank 2 and admits a rational Lagrangian fibration such that the corresponding isotropic class in $H^2(X, \mathbb Z)$ is indivisible. Then, there is a Brauer class $\theta \in \Br(M_{\sigma}(\mathcal A_Y, v))$ and an exact linear equivalence
    $$D^b(M_{\sigma}(\mathcal A_Y, v), \theta^n) \simeq \mathcal A_Y^{[n]},$$
    where $2n = v^2+2$ is the dimension of $M_{\sigma}(\mathcal A_Y, v)$.
\end{theorem}
As with Corollary \ref{cor:generallang}, the divisibility condition is automatic if $n-1$ is square-free. Again, the condition on the Picard rank may be replaced by a genericity condition, see Remark \ref{rem:genericity_2}.

Let us conclude this introduction by giving a short overview of the structure of this note.
In Section \ref{sec:background}, we recall the results of Arinkin \cite{arinkin} and description of the Tate--Shafarevich group of Lagrangian-fibered Hyperkähler manifolds of K3$^{[n]}$-type by Huybrechts and Mattei \cite{huybrechtsmattei}. Then, in Section \ref{sec:twistedpc} we show that Arinkin's sheaf descends to a twisted sheaf on certain torsors. In Section \ref{sec:equiv}, we conclude that the resulting twisted sheaf induces a Fourier--Mukai equivalence. In Section \ref{sec:identifying_brauer}, we specialize to the twisted compactified Picard varieties studied by \cite{huybrechtsmattei} and identify certain Brauer classes we constructed in Section \ref{sec:twistedpc} with Brauer classes arising from moduli theory. Combining the above, we conclude the proof of Theorem \ref{thm:main_hk_intro} in Section \ref{sec:conclusion_hk}. After a brief interlude on Markman--Mukai lattices in Section \ref{sec:markmanmukai}, we prove Theorem \ref{thm:main_kuznetsov_component_intro} in Section \ref{sec:kuznetsovcomponent}.

\subsection{Acknowledgments}
We would like to thank the organizers of the Bootcamp for the 2025 Algebraic Geometry Summer Research Institute, where the ideas for this project first began. Special thanks go to our group's mentor, Laura Pertusi, as well as its members Amal Mattoo, Lily McBeath, Ana Pavlakovi\'c, Weite Pi, Yu Shen, and Nicol\'as Vilches. We would also like to thank Dominique Mattei and Reinder Meinsma for reading a preliminary version of this draft, as well as Daniel Huybrechts and Alex Perry for many useful discussions. The first author wants to thank Alessio Bottini for his interest in this project.

The first author is grateful for the support provided by the ERC Synergy Grant 854361 HyperK as well as the International Max Planck Research School on Moduli Spaces at the Max Planck Institute for Mathematics in Bonn. The second author is grateful for the support provided by the National Science Foundation under NSF grant DMS-2143271 and NSF RTG grant DMS-1840234.

\subsection{Conventions} If not mentioned otherwise, we work over $\mathbb C$ and with the \'etale topology and cohomology groups.   For a morphism $X \to B$, we let $\mathcal X \in \Sh_{\mathrm{\acute{e}t}}(B)$ denote the corresponding \'etale sheaf of sections. If $\alpha \in \SBr(X)$ is a special Brauer class, cf. Section \ref{sec:background_twistedcompact}, we let $\overline{\alpha} \in \Br(X)$ denote the image of $\alpha$ under the natural map $\SBr(X) \to \Br(X)$.
\section{Background}
\label{sec:background}

\subsection{Good abelian fibrations and Arinkin sheaves}
The aim of this section is to recall the properties of Arinkin's extension of the Poincar\'e line bundle in \cite{arinkin} that allow us to generalize Mukai's equivalence to the twisted setting.
\begin{definition}
    Let $\pi \colon X \to B$ be a morphism of smooth quasi-projective varieties and $X^0 \subset X$ the complement of the singular locus of fibers. The morphism $\pi \colon X \to B$ is a \emph{good abelian fibration} if
    \begin{enumerate}
        \item The morphism $\pi$ is projective with geometrically connected fibers;
        \item The restriction $\pi|_{X^0} \colon X^0 \to B$ is an abelian group scheme.
        \item The group law $\mu_{X^0} \colon X^0 \times_B X^0 \to X^0$ extends to a group action $\mu_X \colon X^0 \times_B X \to X$.
    \end{enumerate}
     Let $0_X \colon B \to X^0$ denote the zero section of the abelian group scheme $\pi|_{X^0} \colon X^0 \to B$.
\end{definition}
In particular, the generic fiber of a good abelian fibration is an abelian variety. Note that if $s \colon B \to X$ is a section of a good abelian fibration, then we have $s(B) \subset X^0$ since both $X$ and $B$ are smooth. We let $\transl_{s} \coloneqq \mu_X(s, -) \colon X \to X$ denote translation by $s$.

It is clear from the definition that any abelian scheme is a good abelian fibration. The following example is the most important for the applications towards the end of this note.
\begin{theorem}
    If $\mathcal C \to B$ is a flat family of integral curves with only planar singularities and smooth total space, then the Altman--Kleiman compactification of the relative Jacobian
    $$\pi \colon \overline{\Pic}^0_{\mathcal C / B} \to B$$ is a good abelian fibration.
\end{theorem}
\begin{proof}
    The first property follows from \cite[Thm.\ (9)]{aik} and \cite[Thm.\ 8.5]{ak}. The second comes from the identification of the smooth locus with the Jacobian $\Pic^0_{\cC/B} \to B$ \cite[Thm.\ A(ii)]{mrvsmoothlocus}. The third property holds as the group law on $\Pic^0_{\cC/B}$ corresponds to the tensor product of line bundles, which naturally extends to an action $\Pic^0_{\cC/B}\times_B \overline{\Pic}^0_{\cC/B} \to \overline{\Pic}^0_{\cC/B}$.
\end{proof}

\begin{definition}
    Let $\pi_X \colon X \to B$ and $\pi_Y \colon Y \to B$ be good abelian fibrations over the same base $B$. Then, an \emph{Arinkin sheaf} on $X \times Y$ is the following datum:
    \begin{enumerate}
        \item A sheaf $\mathcal P$ on $X \times Y$ with support on $X \times_B Y$.
        \item \emph{(Normalization)} Isomorphisms
              $$\mathrm{triv}_{X} \colon (\mathrm{id}_{X} \times 0_{Y})^* \mathcal P \congpf \mathcal O_{X}\text{ and } \mathrm{triv}_{Y} \colon (0_X \times \mathrm{id}_{Y})^* \mathcal P \congpf \mathcal O_{Y}$$
        \item \emph{(Theorem of the square)} Isomorphisms
              $$\mathrm{sq}_{X} \colon (\mu \times \mathrm{id}_{Y})^* \cP\congpf \mathrm{pr}_{13}^* \cP \otimes \mathrm{pr}_{23}^* \cP$$
              on $X^0 \times X \times Y$ and
              $$\mathrm{sq}_{Y} \colon (\mathrm{id}_{Y} \times \mu)^* \cP\congpf \mathrm{pr}_{12}^* \cP \otimes \mathrm{pr}_{13}^* \cP$$
              on $X \times {Y}^0 \times Y$
        \item If $U \to B$ is \'etale and $s \colon U \to X|_U$ and $t \colon U \to Y|_U$ are sections, then $(s \times \mathrm{id}_Y)^*\cP$ is a line bundle on $U \times_B Y$ and $(\mathrm{id}_X \times t)^* \cP$ is a line bundle on $X \times_B U$.
        \item The diagrams
              $$\begin{tikzcd}
                      \mathcal O_{X \times X} \arrow[r] \arrow[d, "\mathrm{id}"] & {\mathrm{pr}_{1}^*(\mathrm{id}_X} \times 0_{Y})^*\mathcal P \otimes \mathrm{pr}_{2}^* (\mathrm{id_X} \times 0_{Y})^*\mathcal P \arrow[d] \\
                      \mathcal O_{X \times X} \arrow[r]  & (\mu_X \times 0_{Y})^* \mathcal P,
                  \end{tikzcd}$$and$$\begin{tikzcd}
                      \mathcal O_{Y \times Y} \arrow[r] \arrow[d, "\mathrm{id}"] & {\mathrm{pr}_{1}^*(0_X \times \mathrm{id}_{Y})\mathcal P \otimes \mathrm{pr}_{2}^* (0_X \times \mathrm{id}_{Y})^*\mathcal P} \arrow[d] \\
                      \mathcal O_{Y \times Y} \arrow[r]  & (0_X \times \mu_{Y})^* \mathcal P,
                  \end{tikzcd}$$
              induced by the normalization and 'theorem of the square' isomorphisms commute.
        \item The induced Fourier--Mukai transform $$\Phi_\cP \colon \Db(X) \to \Db(Y)$$ is a $B$-linear equivalence of categories.
    \end{enumerate}
\end{definition}

If $A \to B$ is an abelian scheme, then the Poincar\'e sheaf on $A \times \Pic^0_{A/B}$ is an Arinkin sheaf, see e.g., \cite[Ch.\ 16]{thebook2}. The motivation for the definition is the following:
\begin{theorem}[Arinkin]
    If $\mathcal C \to B$ is  flat family of integral curves with only planar singularities, then there is an Arinkin sheaf $\mathcal P$ on $\overline{\Pic}_{\mathcal C / B}^0 \times \overline{\Pic}^0_{\mathcal C / B}$.
\end{theorem}
\begin{proof}
    Properties (1), (3) and (6) are immediate by \cite[Thm.\ A and C]{arinkin}. Since any section of ${\overline{\Pic}_{\cC/B}^0}$ lies in the smooth locus (in particular, the identity section as well), (2), (4) and (5) follow from the fiberwise description of Arinkin's sheaf.
\end{proof}

If $X$ and $Y$ are good abelian fibrations over $B$ and $\mathcal P$ is an Arinkin sheaf on $X \times Y$, then there is a morphism
\begin{align*}
    \mathcal X            & \to \mathcal Pic_{Y / B}                             \\
    (s \colon U \to X|_U) & \mapsto (s \times \mathrm{id})^*\cP \in \Pic(Y|_{U})
\end{align*}
of \'etale sheaves, which generalizes the natural morphism
\begin{align*}\check{\mathcal A} \simeq \mathcal Pic^0_{A/B} \hookrightarrow \mathcal Pic_{A/B}\end{align*}
for an abelian scheme $A \to B$.

Let us conclude this section by highlighting consequences of the theorem of the square, which will be used frequently in the remainder of this note.
\begin{lemma}\label{lem:square_identities}
    Let $X \to B$ and $Y \to B$ denote good abelian fibrations equipped with an Arinkin sheaf $\cP$. Furthermore, let $s, s' \colon B \to X$ and $t \colon B \to X$ denote sections. Then we have the following identities:
    \[
    \begin{aligned}
        ((s + s') \times \mathrm{id})^* \cP    & \simeq(s \times \mathrm{id})^* \cP \otimes (s' \times \mathrm{id})^* \cP                                                       \\
        (ns \times t)^* \cP                  & \simeq((s \times t)^* \cP)^{\otimes n}                                   \simeq(s \times nt)^* \cP.                           \\
        (\transl_s \times \mathrm{id})^* \cP & \simeq\cP \otimes \pr_2^*(s \times \mathrm{id})^* \cP .                                                                       \\
        (\transl_s \times \transl_t)^* \cP   & \simeq\cP \otimes \pr_1^*(\mathrm{id} \times t)^* \cP  \otimes \pr_2^*(s \times \mathrm{id})^* \cP \otimes \pi_B^*(s, t)^* \cP.
    \end{aligned}
    \]
\end{lemma}

\subsection{Twisted sheaves}
\label{sec:gerbe} In this section, we recall the description of twisted sheaves that will be used throughout this article. Let $\pi \colon X \to B$ be a good abelian fibration. Fix an \'etale cover $\{U_i\}$ of $B$ and sections $s_{ij} \colon U_{ij} \coloneqq U_i \cap U_j \to X \times_B {U_{ij}}$, defining a torsor $s \coloneqq \{s_{ij}\} \in H^1(B, \mathcal X)$. Let $$\pi_s \colon X_s \to B$$ denote the corresponding twist and $\varphi_{ij} = \tau(s_{ij}, -) \colon X \times_B U_{ij} \to X \times_B U_{ji}$ the induced transition maps.\footnote{Note that $X_s$ is not necessarily a scheme, but only an algebraic space. However, in the applications discussed from Section \ref{sec:identifying_brauer} onwards, all twists we consider are actual schemes.}

The Leray spectral sequence for $\pi_s$ and $\mathbb G_m$ induces an exact sequence
$$\Br(B) \to F^1\Br(X_s) \to H^1(B, \mathcal Pic_{X_s / B}) \to H^3(B, \mathbb G_m),$$
where
$$F^1 \Br(X_s) \coloneqq \ker(\Br(X_s) \to H^0(B, R^2(\pi_s)_* \mathbb G_m)$$
is the subgroup of fiber-wise trivial Brauer classes.

\begin{proposition}
    \label{prop:gerbedesc}
    A collection of line bundles $L_{ij} \in \Pic(X \times_B U_{ij})$ and
    \begin{enumerate}
        \item[\rm{(C1)}] nowhere vanishing sections $\theta_{ij} \in H^0(X \times_B {U_{ij}}, L_{ij} \otimes \varphi_{ij}^* L_{ji})$,
        \item[\rm{(C2)}] nowhere vanishing sections $\theta_{ijk} \in H^0(X \times_B {U_{ijk}}, L_{ik} \otimes \varphi_{ij}^* L_{jk} \otimes \varphi_{ik}^* L_{ik}),$
        \item[\rm{(C3)}] satisfying the cocycle condition
              $$\prod_{a, b, c} \varphi^*_{ia} \theta_{abc} = \prod_{a,b} \varphi_{ia}^* \theta_{ab} \in H^0(X \times_B {U_{ijkl}}, \bigotimes_{a,b} \varphi^*_{ia} L_{ab})$$
    \end{enumerate}
    induces a gerbe on $X_s$ with associated Brauer class $$\alpha \in F^1 \Br(X_s) = \ker(\Br(X_s) \to H^0(B, R^2\pi_{s*} \mathbb G_m)).$$
    Furthermore, a collection of sheaves $\mathcal F_i$ on $X \times_B U_{i}$ together with isomorphisms
    $$\psi_{ij} \colon \mathcal F_i|_{X \times_B U_{ij}} \congpf \varphi^*_{ij} \mathcal F_j|_{X \times_B U_{ji}}\otimes L_{ij}$$ satisfying the obvious cocycle conditions compatible with $\theta_{ij}$ and $\theta_{ijk}$ describes an $\alpha$-twisted sheaf on $X_s$.
\end{proposition}
\begin{proof}
    This is a special case of the description of gerbes \'a la Hitchin \cite{hitchin}.
\end{proof}
The map $F^1 \Br(X_s) \to H^1(B, \mathcal Pic_{X_s/B})$ induced by the Leray spectral sequence admits the following description:
\begin{lemma}
    Let $\{L_{ij}\}$ be a collection of line bundles satisfying the conditions of Proposition \ref{prop:gerbedesc}. Let $\alpha \in F^1 \Br(X_s)$ denote the corresponding Brauer class. The image of $\alpha$ via
    $$F^1 \Br(X_s) \to H^1(B, \mathcal Pic_{X_s/B})$$
    is represented by the cocycle $\{[L_{ij}] \in \Pic(X_s \times_B U_{ij}) / \Pic(U_{ij})\}$.
\end{lemma}

\begin{proof}
    See \cite[Lem.\ 5, Sec.\ 3.4 \& 3.5]{explicitleray}.
\end{proof}

\subsection{Twisted compactified Jacobians}
\label{sec:background_twistedcompact}.

Let $(S, L)$ be a polarized K3 surface of genus $g$ and assume that every curve in the linear system $|L|$ is integral. Let $\overline{\Pic}^0 \coloneqq \overline{\Pic}^0(\mathcal C / |L|) \to |L|$ denote the relative compactified Jacobian of the universal curve, which is a Lagrangian-fibered Hyperkähler manifold. Note that since all members of $|L|$ are integral, the Altman--Kleiman compactification of $\Pic^0(\mathcal C / |L|_{\mathrm{sm}})$ agrees with Simpson's compactification, cf.\ \cite[p.\ 9]{simpsonjacobian} and \cite[p.\ 8]{sawontwistedFM}. The Tate--Shafarevich group parametrizing twists of $\overline{\Pic}^0 \to |L|$ was studied in \cite{markmanlagrangian, abashevarogov, abasheva}, see \cite[Sec.\ 3.1]{bottini} for an exposition. For the purpose of this note, we recall the description of the Tate--Shafarevich group given in \cite{huybrechtsmattei, huybrechtsmattei2}.

The \emph{special Brauer group} of the K3 surface $S$, i.e.,
$$\SBr(S) \coloneqq \{\cA \text{ Azumaya algebra on } S\} / (\mathcal A \sim \cA' \otimes \mathcal End(F) \text{ for } F \text{ with } \det(F) \simeq \mathcal O_X),$$
is naturally identified with $H^2(S, \mathbb Q/ \mathbb  Z)$ via the exponential sequence. Let $\SBr(S, L)$ denote the annihilator of $L \in \NS(S)$ for the intersection pairing $\SBr(S) \times \NS(S) \to \mathbb Q / \mathbb Z$. Huybrechts and Mattei show that, for all $\alpha \in \SBr(S, L)$,  the moduli space $\overline{\Pic}^0_{\alpha}(\mathcal C / |L|) = M_{\alpha}(0, L, g-1)$ of $\alpha$-twisted sheaves\footnote{As all members of $|L|$ are integral, one does not need to worry about stability at this point.} of rank one and degree zero is non-empty. Moreover, using \cite{markmanlagrangian}, they show that any projective twist of $\overline{\Pic}^0 \to |L|$ is given by $\overline{\Pic}^0_{\alpha}$ for some $\alpha \in \SBr(S, L)$, see \cite[Thm.\ 1.1]{huybrechtsmattei}. In the language introduced in the previous sections applied to $\overline{\Pic}^0$ and the Arinkin sheaf introduced in \cite{arinkin}, this may be rephrased by saying that there is a surjective map
$$\SBr(S, L) \to \{s \in H^1(|L|, \overline{\mathcal Pic}^0 ) \mid (\overline{\Pic}^0)_s \text{ is projective}\}.$$
\section{Twisting Arinkin sheaves}
\label{sec:twistedpc}
The aim of the next two sections is to prove the following theorem by refining the strategy of the proof of \cite[Prop.\ 3.3]{bottini}:
\begin{theorem}
    \label{thm:main_abelian_schemes}
    Let $B$ a non-singular quasi-projective variety and let $\pi_X \colon X \to B$ and $\pi_Y \colon Y \to B$ denote good abelian fibrations such that there is an Arinkin sheaf $\cP$ on $X \times Y$. Let $X_s$ be a twist of $X$ and $Y_t$ be a twist of $Y$ such that $s \in H^1(B, \mathcal X)$ is $n$-torsion and $t \in H^1(B, \mathcal Y)$ is $n$-divisible. Then, there are Brauer classes $o_s(t) \in \Br(X_s)$ and $o_t'(s) \in \Br(Y_t)$ and exact linear equivalences
    $$D^b(X, o_{s}(t)^{-1} \cdot \pi_{X_s}^*\gamma)) \simeq D^b(Y,  o_t'(s) \cdot \pi_{Y_t}^* \gamma)$$
    for all $\gamma \in \Br(B).$
\end{theorem}

\begin{remark}\label{rem:benbassat}
    This can be viewed as a variant of \cite[Cor.\ 6.2]{benbassat}, which follows from a more general formalism for twisting derived equivalences. While less general, our result differs in that we need not assume that the Brauer group $\Br(B)$ is trivial, and our description of the Brauer classes is somewhat more explicit.
\end{remark}

Let $\pi_X \colon X \to B$ and $\pi_Y \colon Y \to B$ be good abelian fibrations. Assume that there is an Arinkin sheaf $\cP$ on $X \times Y$. Let $X_s$ be a twist of $X$ corresponding to an $n$-torsion class $ s \in H^1(B, \mathcal X)$ and $Y_t$ a twist of $Y$ corresponding to an $n$-divisible class in $H^1(B, \mathcal Y)$. In particular, we find an \'etale cover $\{U_i\}$ of $B$ and representatives $\{s_{ij} \in H^0(U_{ij}, \mathcal X)\} = s \in H^1(B, \mathcal X)$ and $s_i' \in H^0(U_i, \mathcal X)$ such that
$$ns_{ij} = s_i'|_{U_{ij}} - s_j'|_{U_{ij}}$$
and $\{t'_{ij} \in H^0(U_{ij}, \mathcal Y)\} = t' \in H^0(B, \mathcal Y)$ with
$$t = \{t_{ij} \coloneqq nt_{ij}'\} \in H^1(B, \mathcal Y).$$

In order to descend the Arinkin sheaf $\cP$ to a twisted sheaf on $X_s \times Y_t$, we need to find a collection of line bundles $M_{ij} \in \Pic(X|_{U_{ij}})$, $N_{ij} \in \Pic(Y|_{U_{ij}})$
satisfying
\begin{equation}\label{eq:goal}(\transl_{s_{ij}} \times \transl_{t_{ij}})^* \cP \simeq \cP \otimes \pr_1^* M_{ij} \otimes \pr_2^* N_{ij}.\end{equation}
First, note that the Arinkin sheaf pulls back under $(\transl_{s_{ij}} \times \transl_{t_{ij}})$ as follows:
\begin{lemma}\label{lem:trtrP1}
    We have
    $$(\transl_{s_{ij}} \times \transl_{t_{ij}})^* \mathcal P \simeq \cP \otimes \pr_2^*(s_{ij} \times\mathrm{id})^* \cP \otimes \pr_1^*(\mathrm{id}\times t_{ij})^* \cP \otimes \pi_B^* (s_{ij} \times t_{ij})^* \cP.$$
\end{lemma}
\begin{proof}
    Applying the fourth identity of Lemma \ref{lem:square_identities}, we obtain
    \begin{align*}
        (\transl_{s_{ij}} \times \transl_{t_{ij}})^* \mathcal P & \simeq (\transl_{s_{ij}} \times \mathrm{id})^*(\mathrm{id} \times \transl_{t_{ij}})^* \mathcal P                                                       \\
                                                                & \simeq(\transl_{s_{ij}} \times \mathrm{id})^*(\cP \otimes \pr_1^*(\mathrm{id} \times t_{ij})^* \mathcal P)                                             \\
                                                                & \simeq \cP \otimes \pr_2^*(s_{ij} \times\mathrm{id})^* \cP \otimes \pr_1^*(\mathrm{id}\times t_{ij})^* \cP \otimes \pi^* (s_{ij} \times t_{ij})^* \cP,
    \end{align*}
    as desired.
\end{proof}
Note the similarity between (\ref{eq:goal}) and Lemma \ref{lem:trtrP1}. The following allows us to distribute the extraneous term $\pi^*(s_{ij} \times t_{ij})^* \cP$ over the factors.
\begin{lemma}\label{lem:trtrP2}
    We have
    $$\pi^*(s_{ij} \times t_{ij})^* \cP \simeq \pr_1^* \pi^*(s_i' \times t'_{ij})^* \cP \otimes \pr_2^* \pi_B^*(s_j' \times t'_{ji})^* \cP.$$
\end{lemma}
\begin{proof}
    Applying the first and second identities of Lemma \ref{lem:square_identities}, we obtain
    \begin{align*}
        \pi^*(s_{ij} \times t_{ij})^* \cP & \simeq \pi^*(s_{ij} \times nt_{ij}')^* \cP                                                  \\
                                          & \simeq \pi^*(n s_{ij} \times t'_{ij})^* \cP                                                 \\
                                          & \simeq \pi^*((s_i'-s_j') \times t_{ij}')^* \cP                                              \\
                                          & \simeq \pr_1^* \pi^*(s_i' \times t'_{ij})^* \cP \otimes \pr_2^*(s_j' \times t'_{ji})^* \cP,
    \end{align*}
    as desired.
\end{proof}
Therefore, it makes sense to consider the collections of line bundles
$$M_{ij} \coloneqq (\mathrm{id} \times t_{ij})^* \mathcal P \otimes \pi^*(s_i' \times t_{ij}')^* \mathcal P \in \Pic(X|_{U_{ij}})$$
and
$$N_{ij} \coloneqq (s_{ij} \times \mathrm{id})^* \mathcal P \otimes \pi^* (s_j' \times t_{ji}')^* \mathcal P \in \Pic(Y|_{U_{ij}}).$$
Combining Lemma \ref{lem:trtrP1} and Lemma \ref{lem:trtrP}, we conclude:
\begin{corollary}\label{lem:trtrP}
    We have
    $$(\transl_{s_{ij}} \times \transl_{t_{ij}})^* \mathcal P \simeq \mathcal P  \otimes \pr_1^* M_{ij} \otimes \pr_2^* N_{ij}.$$
\end{corollary}

In view of the discussion in Section \ref{sec:gerbe}, it remains to show that the collections $\{M_{ij}\}$ and $\{N_{ij}\}$ satisfy the conditions listed in Proposition \ref{prop:gerbedesc}.

\begin{lemma}\label{lem:alphagerbe}
    The collection of line bundles $\{M_{ij}\}$ defines a Brauer class $o_s(t) \in H^2(X_s, \bbG_m)$.
\end{lemma}

\begin{proof}
    Applying repeatedly the identities of Lemma \ref{lem:square_identities} and the facts that $ns_{ij} = s_i'-s_j'$ and $nt_{ij}' = t_{ij}$, we obtain isomorphisms
    \begin{align*}\varphi_{ijk} \colon \transl_{s_{ij}}^*M_{jk} &\congpf (\mathrm{id} \times t_{jk})^* \mathcal P  \otimes \pi^*((s_{ij} \times t_{jk})^* \mathcal P  \otimes (s_j' \times t_{jk}'))^* \mathcal P \\&\congpf (\mathrm{id} \times t_{jk})^* \mathcal P \otimes \pi^*(s_i' \times t_{jk}')^* \mathcal P\end{align*}
    The nowhere vanishing sections $\theta_{ij} \in H^0(X|_{U_{ij}}, M_{ij} \otimes \transl_{s_{ij}}^* M_{ji})$ are defined as:
    \begin{align*}
        \mathcal O_A & \xrightarrow{\mathrm{trivialization}} (\mathrm{id} \times (t_{ij} + t_{ji}))^* \cP \otimes \pi^*(s_i' \times (t_{ij}'+t_{ji}'))^* \cP                                                                                  \\
                     & \xrightarrow{\mathrm{square}}(\mathrm{id} \times t_{ij})^* \mathcal P \otimes \pi^*(s_i' \times t_{ij}')^* \mathcal P \otimes (\mathrm{id} \times t_{ji})^* \mathcal P \otimes \pi^*(s_i' \times t_{ji}')^* \mathcal P \\
                     & \xrightarrow{\varphi_{iij}^{-1} \otimes \varphi_{iji}^{-1}} M_{ij} \otimes \transl^*_{s_{ij}} M_{ji}
    \end{align*}
    A similar process yields nowhere vanishing sections
    $$\theta_{ijk} \in H^0(X|_{U_{ijk}}M_{ij} \otimes \transl_{ij}^* M_{jk} \otimes \transl_{ik}^* M_{ki}):$$
    \begin{align*}
        \mathcal O_A & \xrightarrow{\mathrm{trivialization}} (\mathrm{id} \times (t_{ij} + t_{jk} + t_{ki}))^* \cP \otimes \pi^*(s_i' \times (t_{ij}'+t_{jk}'+t_{ki}'))^* \cP                                                                                 \\
                     & \xrightarrow{\mathrm{square}}(\mathrm{id} \times t_{ij})^* \mathcal P \otimes \pi^*(s_i' \times t_{ij}')^* \mathcal P \otimes \cdots \otimes  (\mathrm{id} \times t_{ki})^* \mathcal P \otimes \pi^*(s_i' \times t_{ki}')^* \mathcal P \\
                     & \xrightarrow{\varphi_{iij}^{-1} \otimes \varphi_{ijk}^{-1} \otimes  \varphi_{iki}^{-1}} M_{ij} \otimes \transl^*_{s_{ij}} M_{jk} \otimes \transl^*_{s_{ik}} M_{ki}
    \end{align*}
    It remains to check the cocycle conditions. For this, we note that the compability between our chosen trivialization of $\cP$ along the identity sections and the theorem of the square (property (5) of the Arinkin sheaf $\cP$) implies that the section $\transl_{s_{ij}}^* \theta_{jk} \in H^0(X|_{U_{ijk}}, \transl^*_{s_{ij}} M_{jk} \otimes \transl_{s_{ik}}^* M_{kj})$ agrees with the composition
    \begin{align*}
        \mathcal O_A & \xrightarrow{\mathrm{trivialization}} (\mathrm{id} \times (t_{jk} + t_{kj}))^* \cP \otimes \pi^*(s_i' \times (t_{jk}'+t_{kj}'))^* \cP                                                                                  \\
                     & \xrightarrow{\mathrm{square}}(\mathrm{id} \times t_{jk})^* \mathcal P \otimes \pi^*(s_i' \times t_{jk}')^* \mathcal P \otimes (\mathrm{id} \times t_{kj})^* \mathcal P \otimes \pi^*(s_i' \times t_{kj}')^* \mathcal P \\
                     & \xrightarrow{\varphi_{ijk}^{-1} \otimes \varphi_{ikj}^{-1}} \transl_{s_{ij}}^*M_{jk} \otimes \transl^*_{s_{ik}} M_{kj}
    \end{align*}
    while the section $\transl_{s_ij}^* \theta_{jkl} \in H^0(X|_{U_{ijkl}}, \transl^*_{s_{ij}} M_{jk} \otimes \transl_{s_{ik}}^* M_{kl} \otimes \transl_{s_{il}}^* M_{lj})$ agrees with the composition
    \begin{align*}
        \mathcal O_A & \xrightarrow{\mathrm{trivialization}} (\mathrm{id} \times (t_{jk} + t_{kl} + t_{lj}))^* \cP \otimes \pi^*(s_i' \times (t_{jk}'+t_{kl}'+t_{lj}'))^* \cP                                                                                 \\
                     & \xrightarrow{\mathrm{square}}(\mathrm{id} \times t_{jk})^* \mathcal P \otimes \pi^*(s_i' \times t_{jk}')^* \mathcal P \otimes \cdots \otimes  (\mathrm{id} \times t_{lj})^* \mathcal P \otimes \pi^*(s_i' \times t_{lj}')^* \mathcal P \\
                     & \xrightarrow{\varphi_{ijk}^{-1} \otimes \varphi_{ikl}^{-1} \otimes  \varphi_{ilj}^{-1}} \transl_{ij}^*M_{jk} \otimes \transl^*_{s_{ik}} M_{kl} \otimes \transl^*_{s_{il}} M_{jl}
    \end{align*}
    It is then straightforward to check that the cocycle computation follows by the compatibility of property (5) of the Arinkin sheaf $\cP$.
\end{proof}

The proof of the following lemma is very similar to the previous one. Due to the slight asymmetry in the setup, we give it anyways.
\begin{lemma}
    The collection of line bundles $\{N_{ij}\}$ defines a Brauer class $o_t'(s) \in H^2(Y_t, \bbG_m)$.
\end{lemma}
\begin{proof}
    Apply the strategy above, but use the following identification instead:
    \begin{align*}\psi_{ijk} \colon \transl_{t_{ij}}^*N_{jk} &\congpf (s_{jk} \times \mathrm{id})^* \mathcal P \otimes \pi^*((s_{jk} \times t_{ij})^* \mathcal P  \otimes (s_j' \times t_{kj}'))^* \mathcal P \\&\congpf (s_{jk} \times \mathrm{id})^* \mathcal P \otimes \pi^*((s_j' \times t_{ij}')^* \mathcal P \otimes (s'_{k} \times t_{ji}')^* \cP \otimes (s_j' \times t'_{kj})^* \cP )\end{align*}
    Then, continue as in the proof of Lemma \ref{lem:alphagerbe}.
\end{proof}

All in all, we have managed to descend the Arinkin sheaf to a twisted sheaf:

\begin{lemma} \label{lem:twistedpoincare}
    The Arinkin sheaf descends to a $(o_s(t)\boxtimes o'_t(s))$-twisted sheaf $^{{s,t}}\cP$ on $X_s\times Y_t$.
\end{lemma}
\begin{proof}
    Follows from Lemma \ref{lem:trtrP} and the description of twisted sheaves in Proposition \ref{prop:gerbedesc}.
\end{proof}

Note that $o_s(t)$ and $o_t(s)$ depend on the choice of trivialization datum $s'$ and as well as the $n$-th root $t'$. Different choices of $s'$ and $t'$ amount to changing both $o_s(t)$ and $o_t'(s)$ by a Brauer class pulled back from the base. In fact, we have the following:
\begin{lemma}\label{lem:grouphom}
    If $s \in H^1(B, \mathcal X)$ is $n$-torsion, then the map
    $$o_s \colon H^1(B, \mathcal Y)_{n-{\mathrm{div}}} \to \Br(X_s) / \Br(B).$$
    is a well-defined group homomorphism. If $s \in H^1(B, \mathcal X)$ is $n$-divisible, then the map
    $$o'_s(t) \colon H^1(B, \mathcal Y)[n] \to \Br(X_s) / \Br(B)$$ is a well-defined group homomorphism. 
    If $s \in H^1(B, \mathcal X)$ and $t \in H^1(B, \mathcal Y)$ are both $n$-torsion and $n$-divisible, then we have
    $$o_s(t) \equiv o_s'(t) \mod \Br(B).$$
\end{lemma}
\begin{proof}
    From the exact sequence
    $$\Br(B) \to F^1\Br(X_s) \to H^1(B, \mathcal Pic_{X_s})$$
    and the fact that
    $$[M_{ij}] \equiv [(\mathrm{id} \times t_{ij})^* \cP]\in H^0(U_{ij}, \mathcal Pic_{X_s})$$
    holds by construction, it follows that the map
    $$[o_s(t)] \in \Br(Y_s) / \Br(B)$$
   is independent of the additional choices. The theorem of the square then implies that it is a group homomorphism. A similar argument yields the remaining claims.
\end{proof}

\section{Twisted Fourier--Mukai transforms}
\label{sec:equiv}
In this section, we finish the proof of Theorem \ref{thm:main_abelian_schemes} by showing that the twisted Arinkin sheaf constructed in the previous section induces a Fourier--Mukai equivalence.

\begin{proposition}
    \label{prop:fullyfaithful}
    The twisted Fourier--Mukai transform
    $$\Phi_{{^{s,t} \cP}} \colon D^b(X_s, o_s(t)^{-1}) \to D^b(Y_t, o'_t(s))$$
    is an equivalence.
\end{proposition}
\begin{proof}
        By construction, the twisted sheaf $^{s,t}\cP$ is supported on $X_s \times_B Y_t \subset X_s \times Y_t$. In particular, $F \coloneq \Phi_{{^{s,t} \cP}}$ defines a $B$-linear functor in the sense of \cite{noncommutativehpd}. If we let $G: D^b(Y_t, o_t(s)) \to D^b(X_s, o_s'(t)^{-1})$ denote its right adjoint, then $G$ is canonically a $B$-linear functor \cite[Lem.\ 2.11]{noncommutativehpd}. The unit and counit for this adjunction then give natural transformations of $B$-linear functors
    \[\mathrm{id}_{\Db(X_s,-o_s(t))} \to GF,\qquad FG \to \mathrm{id}_{\Db(Y_t,o_t(s))}.\]
    To check that the functor $F$ is an equivalence, it suffices to observe that \'etale-locally over the base $B$, the Fourier--Mukai kernel $^{s,t}\cP$ coincides with an untwisted Arinkin sheaf on $X \times_B Y$. Since the untwisted Arinkin sheaf induces a $B$-linear equivalence, it follows that over this \'etale open the cones of the unit and counit both vanish. Hence the cones must vanish globally over the whole base, so that $F$ and $G$ are quasi-inverse to one another.
\end{proof}

All in all, we have concluded the proof of Theorem \ref{thm:main_abelian_schemes} in the case $\gamma = 1 \in \Br(B)$. The general statement follows by observing that the Fourier--Mukai kernel ${}^{s, t} \cP$ is supported on $X_s \times_B Y_t$ and can thus also be endowed with the structure of a $\pr_1^*((o_s(t) \cdot \pi_{X_s}^* \gamma)^{-1}) \otimes \pr_2^*(o'_t(s) \cdot \pi^*_{Y_t} \gamma)$-twisted sheaf which is \'etale locally on $B$ isomorphic to the untwisted Poincar\'e sheaf. The argument in the proof of Proposition \ref{prop:fullyfaithful} then shows that the induced $B$-linear functor is an equivalence.

\section{Identifying certain Brauer classes}\label{sec:identifying_brauer}

In this section, we identify certain instances of the Brauer classes constructed in the previous sections in the setting of relative Jacobians of smooth curves.

Let $\mathcal C \to B$ be a family of proper smooth curves with smooth total space and base. Let $\alpha \in \SBr(\mathcal C)$ be a special Brauer class and let $\Pic^0_\alpha \coloneqq \Pic^0_\alpha(\mathcal C/ B)$ denote the relative moduli space of $\alpha$-twisted line bundles on fibers of $\mathcal C \to B$, see e.g. \cite{simpson}. In the following, assume that $\Pic^0_{\alpha}$ is non-empty. In particular, this implies that the restriction $\alpha|_{\mathcal C_t} = 0 \in \SBr(\mathcal C_t)$ to geometric fibers is trivial, cf. \cite[Rem.\ 3.3]{huybrechtsmattei}. Furthermore, $\Pic^0_\alpha$ is naturally a torsor under the abelian scheme $\Pic^0(\mathcal C/ B) \to B$.

Explicitly, this torsor can be described as follows: Fix an Azumaya algebra $\mathcal A$  representing $\alpha$. Since the moduli space $\Pic^0_\alpha$ is assumed to be non-empty, we find an \'etale cover $\{U_i\}$ of $B$, sections $c_i \colon U_i \to \mathcal C|_{U_i}$
and locally free $\mathcal A|_{U_i}$-modules $\cQ_i$ of rank one and degree zero. The untwisted line bundles $\cQ_i \otimes_{\cA|_{U_{ij}}} \cQ_j^\vee$ yield sections $\psi_{ij} \colon U_{ij} \to \Pic^0|_{U_{ij}}$ that define a $2$-cocycle representing $\Pic^0_{\alpha} \in H^1(B, \mathcal Pic^0_{\cC/B})$. Let $\psi_i \colon \Pic^0_\alpha|_{U_i} \to \Pic^0|_{U_i}$ denote the trivialization morphism corresponding to $- \otimes_{\mathcal A} \cQ_i^\vee$ with inverse $\psi_i^{-1} \colon \Pic^0|_{U_i} \to \Pic^0_{\alpha}|_{U_i}$ corresponding to $- \otimes \cQ_i$. In particular, the composition $\psi_j \circ \psi_i^{-1}$ is given by translation by $\psi_{ij}$.

Assuming that $\alpha \in \SBr(\mathcal C)$ is sufficiently divisible so that we can apply Theorem \ref{thm:main_abelian_schemes}, there are two natural constructions of twisted sheaves on the product $\cC \times \Pic^0_\alpha$: Since $\Pic^0_{\alpha}$ is a coarse moduli space of $\mathcal A$-modules, there exists a (twisted) universal object
$^{\alpha}\mathcal L$ on $\mathcal C \times \Pic^0_{\alpha}$, which we can interpret as an $\alpha \boxtimes \theta$-twisted sheaf, where $\theta$ is the obstruction to gluing the local universal objects (which do indeed exist since locally over the base, $\Pic^0_{\alpha}$ is a fine moduli space). On the other hand, we constructed an $o_1(\alpha) \boxtimes o_{\alpha}(1)$-twisted Poincar\'e sheaf $^{1, \alpha} \mathcal P$ on $\Pic^1 \times \Pic^0_{\alpha}$, which we may pull back along the Abel--Jacobi map $\cC \times \Pic_\alpha^0 \to \Pic^1 \times \Pic_\alpha^0.$

The aim of this sections is to compare the twists $\alpha$ and $o_1(\alpha)$, as well as the class $o_{\alpha}(1)$ and the obstruction class $\theta$. Without fixing the additional choices necessary for our construction, the classes $o_1(\alpha)$ and  $o_\alpha(1)$ are only well-defined up to twisting by Brauer classes pulled back from the base $B$. It is thus natural to compare the resulting Brauer classes just modulo $\Br(B)$.
On that note, recall that for a fibration $f \colon X \to B$, the Leray spectral sequence induces the exact sequence
$$\Br(B) \to F^1\Br(X) \to H^1(B, \mathcal Pic_{X/B}).$$
In other words, comparing Brauer classes modulo  classes pulled backed from the base is equivalent to comparing their images in $H^1(B, \mathcal Pic_{X/B})$.

Recall that we let $\overline{\alpha} \in \Br(\mathcal C)$ denote the image of $\alpha \in \SBr(\mathcal C)$ under the natural map $\SBr(\mathcal C) \to \Br(\mathcal C)$.

\begin{proposition}\label{lem:twist}
    The images of $o_1(\alpha)$ and $\overline{\alpha}$ in $H^1(B, \mathcal Pic_{\mathcal C})$ coincide.
\end{proposition}

\begin{proof}
    Let us begin by describing the image of $\alpha.$ As mentioned in the beginning of this section, the image of $\alpha$ in $H^1(B, \mathcal Pic_{\cC /B})$ coincides with the torsor induced by $\Pic^0_{\alpha}$ and is thus represented by the $2$-cocycle
    $$\{\cQ_i \otimes_{\cA|_{U_{ij}}} \cQ_j^\vee\}$$

    Let ${}^0 \cL_i \in \Pic(\cC \times_B \Pic^0|_{U_i})$ denote local universal line bundles of degree zero on $\cC$, which exist as $\cC|_{U_i}$ admits a section. Of course, ${}^0 \cL_i$ is well-defined only up to tensoring by lines bundles pulled back from $\Pic^0|_{U_i}$. By definition of the sections $\psi_{ij}$ and the universality of $^{0} \cL_i$, we have
    $$(\mathrm{id} \times \psi_{ij})^*({}^0 \mathcal L_{i}) \equiv \mathcal Q_i \otimes_{\mathcal A_{ij}} \mathcal Q_j^\vee \mod \Pic(B).$$

    On the other hand, the image of $o_1(\alpha)$ is by construction represented by the $2$-cocycle given by
    $$((\mathrm{\transl_{\cO(-c_i)} \circ a})\times \psi_{ij})^*({^{0, 0} \cP)} \simeq (\mathrm{id} \times \psi_{ij})^*({}^0 \mathcal L_i)\otimes \mathrm{pr}_B^*(c_i \times\mathrm{id})^*({}^0 \cL_i)^\vee,$$
    which represents the same element of $H^1(B, \mathcal Pic_{\cC / B}).$
\end{proof}

\begin{proposition}\label{lem:obstruction}
    The image of $o_{\alpha}(1)$ in $H^1(B, \mathcal Pic_{\Pic^0_{\alpha}})$ agrees, up to sign, with the obstruction class to the existence of a universal $\alpha$-twisted sheaf on $\mathcal C \times \Pic^0_{\alpha}$.
\end{proposition}
\begin{proof}
    Let us first describe the image of $o_{\alpha}(1)$. By construction, the image of $o_{\alpha}(1)$ is represented by
    $$(\mathcal O(c_i-c_j) \times \psi_i)^*({}^{0, 0} \cP) \simeq (c_i \times \psi_i)^*({}^0 \cL_i) \otimes (c_j \times \psi_i)^*({}^0 \cL_i)^\vee. $$
    In order to go on, we have to describe the obstruction to the existence of a universal object. Locally, over $U_i$, it can be described as follows:
    $${}^\alpha \cL_i \coloneqq (\mathrm{id} \times \psi_i)^*({}^0 \cL_i) \otimes \pr_1^* \cQ_i,$$
    inheriting the $\pr_1^*\mathcal A$-module structure from $\cQ_i$. Note that this is a universal rank one and degree zero $\mathcal A|_{U_{i}}-$module. The obstruction to glueing the ${}^\alpha \cL_i$ to a global object ${}^\alpha \cL$ is thus given by a Brauer class pulled back from $\Pic^0_{\alpha}$, which can be represented by a collection of line bundles $O^{\alpha}_{ij} \in \Pic(\Pic^0_{\alpha})$, so that
    \[^\alpha\cL_i \simeq {^\alpha\cL_j} \otimes \mathrm{pr}_2^*O_{ij}^\alpha.\]
    By restriction along the section $\Pic_\alpha^0 \xrightarrow{c_i \times \mathrm{id}} \cC \times_B \Pic^0_\alpha$, we have
    $$(c_j \times \mathrm{id})^*({}^{\alpha} \cL_i) \equiv (c_j \times \mathrm{id})^*({}^\alpha \cL_j) \otimes O^{\alpha}_{ij} \mod \Pic(B).$$ We may thus rewrite the line bundles $O_{ij}^\alpha$ as
    \begin{align*}
        O^{\alpha}_{ij} & \equiv(c_j \times \psi_i)^*({}^0 \cL_i) \otimes (c_j \times \psi_j)^*({}^0 \cL_j) ^\vee \otimes (c_j \times \psi_i)^*\mathrm{pr}_1^* (\cQ_i \otimes_\cA \cQ_j^\vee) \\
                        & \equiv  (c_j \times \psi_i)^*({}^0 \cL_i) \otimes (c_j \times \psi_j)^*({}^0 \cL_j)^\vee \otimes \pr_B^* c_j^*(\cQ_i \otimes_{\cA} \cQ_j^\vee)                      \\
                        & \equiv (c_j \times \psi_i)^*({}^0 \cL_i) \otimes (c_j \times \psi_j)^*({}^0 \cL_j)^\vee \mod \Pic(B)
    \end{align*}

    It follows that the image of the obstruction class in $H^1(B, \mathcal Pic_{\Pic^0_{\alpha}})$ is represented by
    $$(c_j \times \psi_i)^*({}^0 \cL_i) \otimes (c_j \times \psi_j)^*({}^0 \cL_j)^\vee \in H^1(B, \mathcal Pic_{\Pic^0_{\alpha}}),$$
    which allows us to conclude.
\end{proof}

Similarly, we constructed an $o_n(\alpha) \boxtimes o_{\alpha}(1)$-twisted Poincar\'e sheaf $^{n, \alpha} \cP$ on $\Pic^n \times \Pic^\alpha$ for any $n \in \mathbb Z$. In the following, fix $n \geq 1$ and let $a^{(n)} \colon \mathcal C^{(n)} \to \Pic^n$ denote the relative Abel--Jacobi map. If $\mathcal L$ is a line bundle on $\mathcal C$, then the line bundle $\prod_{i = 1}^n \pr_{i} \mathcal L$ admits a natural $\mathfrak S_n$-linearization and descends to a line bundle $\mathcal L^{(n)}$ on $\mathcal C^{(n)}$. Similarly, if $\beta \in \Br(\mathcal C)$ is a Brauer class, then the Brauer class $\prod_{i = 1}^n \pr_{i} \beta$ descends to a Brauer class $\beta^{(n)}$, cf. \cite[Sec.\ 3.1]{huybrechts2025periodindexproblemhyperkahlermanifolds}. Note that the two constructions are compatible in the following sense: if $\beta \in \Br(\mathcal C)$ is represented by a collection of line bundles $\{\mathcal L_{ij}\}$ as in Section \ref{sec:gerbe},  then $\beta^{(n)}$ is represented by the collection of line bundles $\{\mathcal L_{ij}^{(n)}\}$. In particular, the diagram
$$\begin{tikzcd}
\Br(\mathcal C) \arrow[r] \arrow[d, "(-)^{(n)}"] & {H^1(B, \mathcal Pic_{\mathcal C})} \arrow[d, "(-)^{(n)}"] \\
F^1\Br(\mathcal C^{(n)}) \arrow[r]               & {H^1(B, \mathcal Pic_{\mathcal C^{(n)}})}                 
\end{tikzcd}$$
commutes.

\begin{proposition}\label{lem:hilbert}
    We have
    $$(a^{(n)})^*o_n(\alpha) \equiv \overline{\alpha}^{(n)}\mod \Br(B).$$
\end{proposition}
\begin{proof}
    By construction, the image of $\overline{\alpha}^{(n)}$ in $H^1(B, \mathcal Pic_{\mathcal C^{(n)}})$ is represented by
    $$\{((\mathrm{id} \times \psi_{ij})^*({}^0 \cL))^{(n)}\} \in H^1(B, \mathcal Pic_{\mathcal C^{(n)}}),$$
    where we again observe that $^0\cL_i$ (and therefore $(^0\cL_i)^{(n)}$) is a line bundle when restricted along the section $\cC^{(n)} \to \Pic^g \times_B \Pic^0.$
    \par
    On the other hand, the image of $(a^{(n)})^*o_n(\alpha)$ in $H^1(B, \mathcal Pic_{\mathcal C^{(n)}})$ is represented by
    $$\{((\tau_{\mathcal O(-nc_i)} \circ a^{(n)}) \times \psi_{ij})^*({}^{0, 0} \cP) \}.$$
    By Arinkin's description of ${}^{0, 0} \cP$ in \cite{arinkin}, we have
    \begin{align*}((\tau_{\mathcal O(-nc_i)} \circ a^{(n)}) \times \psi_{ij})^*({}^{0, 0} \cP) &\equiv ((\mathrm{id} \times \psi_{ij})^*({}^0 \cL_i))^{(n)} \mod \Pic(B),\end{align*}
   which allows us to conclude.
\end{proof}

In order to apply the results of this section to compactifications of $\Pic^0_{\alpha}$, we note the following:

\begin{lemma}\label{lem:restriction_brauer}
    Let $f \colon X \to B$ be a morphism between smooth varieties with geometrically integral fibers. Let $U \subset B$ be an open subscheme.
    Then, the natural map
    $$\Br(X) / \Br(B) \to \Br(f^{-1}(U)) / \Br(U)$$
    is injective.
\end{lemma}
\begin{proof}
    Let $D_i \in B^{(1)}$ denote the divisorial components of the complement of $U \subset B$. The Gysin sequence induces the commutative diagram
    $$\begin{tikzcd}
            0 \arrow[r] & \Br(X) \arrow[r]           & \Br(X|_U) \arrow[r]        & {\oplus H^1(k(f^{-1}(D_i)), \mathbb Q/\mathbb Z)}            \\
            0 \arrow[r] & \Br(B) \arrow[r] \arrow[u, "f^\ast"] & \Br(U) \arrow[u, "f|_U^\ast"] \arrow[r] & {\oplus H^1(k(D_i), \mathbb Q/ \mathbb Z)}. \arrow[u]
        \end{tikzcd}$$
    Assume that we have $\alpha \in \Br(X)$ and $\beta \in \Br(U)$ such that $\alpha|_{X|_U} = f^* \beta$.  We claim that this already implies $\alpha \in f^* \Br(B)$: Indeed, by the diagram above, it suffices to show that the maps
    $$H^1(k(D_i), \mathbb Q / \mathbb Z) \to H^1(k(f^{-1}(D_i)), \mathbb Q / \mathbb Z)$$
    are injective. This immediately follows from the fact that $k(D_i)$ is algebraically closed in $k(f^{-1}(D_i))$, due to the geometric integrality of the fibers of $f$.
\end{proof}

\section{Brauer classes on twisted compactified Jacobians}\label{sec:conclusion_hk}
In this section, we apply the results of the previous sections to twisted compactified Jacobians.
Let $(S, L)$ be a polarized K3 surface of genus $g$ such that all members of the linear system $|L|$ are integral. Fix a special Brauer class $\alpha \in \SBr(S, L)$.
By \cite{arinkin}, the Poincar\'e bundle over the smooth locus extends to an Arinkin sheaf on $\overline{\Pic}^0 \times \overline{\Pic}^0$ and, as $\Br(|L|) = 0$ and $\SBr(S, L)$ is both torsion and divisible, the construction given in Section \ref{sec:twistedpc}, see Lemma \ref{lem:grouphom}, yields a group homomorphism
$$o_{\alpha} \colon \SBr(S, L) \to \Br(\overline{\Pic}^0_{\alpha}).$$
Let $\mathcal C \subset S \times |L|$ denote the universal curve. As the projection $\mathcal C \to S$ is a $\mathbb P^{g-1}$-bundle, we obtain an isomorphism
$$\Br(S) \congpf \Br(\mathcal C).$$
Moreover, the relative Abel--Jacobi map $a \colon \mathcal C \dashrightarrow \overline{\Pic}^1$ induces a homomorphism $$a^* \colon \Br(\overline{\Pic}^1) \to \Br(\mathcal C).$$

\begin{proposition}
    The diagram
    $$
        \begin{tikzcd}
            \SBr(S, L) \arrow[r, "o_1"] \arrow[d, "\beta \mapsto \overline{\beta}"]&\Br(\overline\Pic^1) \arrow[d, "a^*"] \\
            \Br(S) \arrow[r, "\sim"] & \Br(\mathcal C) 
        \end{tikzcd}
    $$
    commutes.
\end{proposition}
\begin{proof}
    The claim follows by combining Lemma \ref{lem:twist}, applied to the smooth locus of the relative curve, with Lemma \ref{lem:restriction_brauer} and the fact that $\Br(|L|) = 0$ is trivial.
\end{proof}

Since $\Br(|L|)$ is trivial and $\SBr(S, L)$ is torsion and divisible, Proposition \ref{prop:fullyfaithful} yields:
\begin{theorem}\label{thm:main_hk}
    Let $\alpha, \beta \in \SBr(S,L)$. Then, there is a twisted derived equivalence
    $$D^b(\overline{\Pic}^0_{\alpha}, o_{\alpha}(\beta)^{-1}) \simeq D^b(\overline\Pic_{\beta}^0, o_{\beta}(\alpha)).$$
\end{theorem}

Recall if we denote the genus of the general member of $|L|$ by $g$, there is a birational map
\begin{equation}
S^{[g]} \overset{\sim} \dashrightarrow \mathcal C^{[g]} \coloneqq\Hilb^g(\mathcal C / |L|)  \overset{\sim}\dashrightarrow \overline{\Pic}^g,\label{eq:birmap}\end{equation}
since $g$ general points on $S$ are contained in a unique curve in the linear system $|L|$,
see, e.g., \cite[Sec.\ 3]{adm}. Let $\pi \colon S^{[g]} \to S^{(g)}$ denote the Hilbert--Chow morphism. Then, the homomorphism
\begin{align*}
Br(S) \to \Br(S^{(g)}) \to  \Br(S^{[g]}),\;
    \beta \mapsto \beta^{(g)} \mapsto \beta^{[g]} \coloneqq \pi^* \beta^{(g)}
\end{align*}
is an isomorphism, see \cite[Lem.\ 3.6]{huybrechts2025periodindexproblemhyperkahlermanifolds}.
\begin{proposition}\label{prop:hk_hilbert}
The isomorphism $\Br(\overline{\Pic}^g) \simeq \Br(S^{[g]})$ induced by (\ref{eq:birmap}) sends $o_{g}(\alpha)$ to $\overline{\alpha}^{[g]}$.
\end{proposition}
\begin{proof}
Note that we have $F^1 \Br(\overline{\Pic}^g) = \Br(\overline{\Pic}^g)$ since $\overline{\Pic}^g$ is a Lagrangian-fibered hyperkähler manifold of K3$^{n}$-type, see \cite[Sec.\ 2.1]{huybrechts2025periodindexproblemhyperkahlermanifolds} and isomorphisms $\Br(S^{[g]}) \simeq \Br(\mathcal C^{[g]}) \simeq \Br(\overline{\Pic}^g)$ induced by the birational maps (\ref{eq:birmap}).

    By construction, the diagram
    $$\begin{tikzcd}
\Br(S) \arrow[r, "\sim"] \arrow[d, "(-)^{[g]}", "\simeq"'] & \Br(\mathcal C) \arrow[r] \arrow[d, "(-)^{[g]}"] & \Br(\mathcal C_{\mathrm{sm}}) \arrow[d, "(-)^{(g)}"] \\
{\Br(S^{[g]})} \arrow[r, "\sim"]             & {\Br(\mathcal C^{[g]})} \arrow[r]   & F^1Br(\mathcal C^{(g)}_{\mathrm{sm}}) 
\end{tikzcd}$$
commutes. The claim then follows by combining Lemma \ref{lem:hilbert}, applied to $\mathcal C_{\mathrm{sm}} / |L|_{\mathrm{sm}}$, with Lemma \ref{lem:restriction_brauer} and the vanishing of $\Br(|L|) = 0$.
\end{proof}

As $\overline{\Pic}^0_{\alpha}$ is usually not a fine moduli space of twisted sheaves on $S$, there is an Brauer class on $\overline{\Pic}^0_\alpha$ obstructing the existence of an universal object. Generalizing \cite[Prop.\ 5.5]{matteimeinsma2025}, we relate the obstruction class to $o_{\alpha}([\Pic^1])$:
\begin{proposition}\label{prop:hk_obstruction}
    The image of $[\Pic^1] \in \SBr(S, L)$ in $\Br(\overline{\Pic}^0_{\alpha})$ agrees, up to sign, with the obstruction to the existence of a universal object on $S \times \overline{\Pic}^0_{\alpha}$.
\end{proposition}
\begin{proof}
    Note that the obstruction class for the existence of a universal object on $S \times \overline{\Pic}^0_{\alpha}$ agrees with the obstruction to the existence of a universal object on $\mathcal C \times \overline{\Pic}^0_{\alpha}.$ Then, as $\Br(|L|) = 0$, the claim follows by combining Lemma \ref{lem:obstruction}, applied to the smooth relative curve $\mathcal C_{\mathrm{sm}} \to |L|_{\mathrm{sm}}$, with Lemma \ref{lem:restriction_brauer} and the fact that $\Br(|L|) = 0$.
\end{proof}

\begin{corollary}\label{cor:twistedpicequiv}
    Let $\alpha \in \SBr(S, L)$. Then, there is a derived equivalence
    $$D^b(\overline{\Pic}^0_{\alpha}, \theta^{-g}) \simeq (D^b(S, \overline{\alpha}))^{[g]},$$
    where $\theta \in \Br(\overline{\Pic}^0_{\alpha})$ is the obstruction for $\overline{\Pic}^0_{\alpha}$ to be a fine moduli space of twisted sheaves on $S$.
\end{corollary}
\begin{proof}
    By Theorem \ref{thm:main_hk}, there is a twisted derived equivalence
    $$D^b(\overline{\Pic}^0_{\alpha}, -o_\alpha(g)) \simeq D^b(\overline{\Pic}^g, o_g(\alpha)).$$
    Note that there is a natural birational map $\Pic^g \overset{\sim}\dashrightarrow S^{[g]}$.
    By Proposition \ref{prop:hk_hilbert} and \cite[Thm.\ 0.3]{dequivalence}, there is a twisted derived equivalence
    $$D^b(\overline{\Pic}^g, o_g(\alpha)) \simeq D^b(S^{[g]}, \overline{\alpha}^{[g]}).$$
    By twisted BKR, cf. \cite[Prop.\ 1.2]{huybrechtsbottini}, we have an equivalence
    $$D^b(S^{[g]}, \overline{\alpha}^{[g]}) \simeq (D^b(S, \overline{\alpha}))^{[g]}.$$
    The claim then follows by composing the above equivalences and noting that we have $o_{\alpha}(g) = g o_{\alpha}(1) = g \theta$, by Proposition \ref{prop:hk_obstruction}.
\end{proof}

\begin{remark}
    Since the right hand-side of the equivalence in Corollary \ref{cor:twistedpicequiv} only depends on the image of $\alpha$ in $\Br(S)$, we obtain an equivalence
    $$D^b(\overline{\Pic}^0_{\alpha}, \theta^{-g}) \simeq D^b(\overline{\Pic}^0_{\alpha'}, \theta'^{-g}),$$
    for all special Brauer classes $\alpha, \alpha' \in \SBr(S, L)$ that map to the same Brauer class in $\Br(S)$. Since the Markman-Mukai lattice of $\overline{\Pic}^0_{\alpha}$ just depends on the image of $\alpha$ in $\Br(S)$, cf. Sec.\ \ref{sec:markmanmukai}, this can be interpreted as a special case of a derived Torelli for hyperkähler manifolds of K3$^{[n]}$-type discussed in \cite{zhang}.
\end{remark}

By the results of Huybrechts and Mattei, any (non-special) Lagrangian-fibered hyperkähler manifold of K3$^{[n]}$-type is birational to a twisted Picard variety  \cite[Thm.\ 1.2]{huybrechtsmattei}. Corollary \ref{cor:twistedpicequiv} may thus be rephrased as follows:
\begin{corollary}\label{cor:generallangrangianfibered}
   Let $f \colon X \to \mathbb P^n$ be a Lagrangian fibered hyperkähler manifold of K3$^{[n]}$-type and assume that $X$ has Picard rank $2$ and $f^* \mathcal O_{\mathbb P^n}(1)$ is indivisible in $H^2(X, \mathbb Z)$. Then, there is a twisted K3 surface $(S, \alpha)$, a Brauer class $\theta \in \Br(X)$ and a linear exact equivalence
    $$D^b(X, \theta^n) \simeq (D^b(S, \alpha))^{[n]}.$$
\end{corollary}

\begin{proof}
By \cite[Thm.\ 1.2]{huybrechtsmattei}, there is a K3 surface $S$ and a complete linear system $\mathcal C \to |L|$ on $S$ and $\alpha \in \SBr(S, L)$ such that $X$ is birational to $\Pic^0_{\alpha}(\mathcal C / |L|_{\mathrm{sm}})$. From $\rho(X) = 2$, one deduces $\rho(S) = 1$ and the assumption on the divisibility of $f^* \mathcal O_{\mathbb P^n}(1)$ ensures that $L$ is primitive by \cite[Thm.\ 1.5]{markmanlagrangian}. In particular, all members of the linear system $|L|$ are integral. The claim thus follows by combining Corollary \ref{cor:twistedpicequiv} with \cite[Thm.\ 0.3]{dequivalence}.
\end{proof}

\begin{remark}\label{rem:genericity}
Since the existence of the polarized K3 surface $(S, L)$ is an open condition by \cite[Thm.\ 1.2]{huybrechtsmattei} and \cite{markmanlagrangian} and the condition that all members of $|L|$ is integral is an open condition on the moduli space of polarized K3 surfaces, the condition on the Picard rank of $X$ can be weakened to a genericity assumption. In other words, the conclusion of Corollary \ref{cor:generallangrangianfibered} holds on a dense open subset of the moduli space of hyperkähler manifolds of K3$^{[n]}$-type with a Lagrangian fibration for which the pullback of $\mathcal O_{\mathbb P^n}(1)$ has divisibility one.
\end{remark}

\section{Markman--Mukai lattices and derived equivalences of K3 categories}\label{sec:markmanmukai}
Let $M$ be a hyperkähler manifold of K3$^{[n]}$-type. By the work of Markman, there is a natural extension of lattices and weight 2 Hodge structures $H^2(M,\bbZ) \subset L(M)$ which has been called the Markman--Mukai lattice \cite[Cor.\ 9.5]{markmanmukailattices}. 

\begin{theorem}[{\cite[Sec.\ 9.1]{markmanmukailattices}}]
    \label{thm:mmlattice}
    For any hyperkähler manifold $M$ of K3$^{[n]}$-type, the Markman--Mukai lattice $L(M)$ satisfies the following properties. 
    \begin{enumerate}[label={\rm{(\alph*)}}]
        \item Abstractly as a lattice, $L(M)$ coincides with the Mukai lattice of a K3 surface.
        \item If $M = M_H(v)$ is a moduli space of stable sheaves on K3 surface $S$ there is a Hodge isometry $L(M) \simeq \widetilde{H}(S,\bbZ)$ which preserves the natural embeddings of $H^2(M,\bbZ)$.
        \item If $n = \frac{1}{2}\dim M \geq 4$, then $L(M)$ is naturally Hodge-isometric to the quotient Hodge structure $Q^4(M,\bbZ) \coloneqq H^4(M,\bbZ)/\mathrm{Sym}^2H^2(M,\bbZ)$ endowed with the unique even unimodular monodromy-invariant bilinear form, and the embedding of $H^2(M) \hookrightarrow L(M)$ identifies $H^2(M)$ with $c_2(M)^\perp$.
        \item If $M$ and $N$ are birational hyperkähler manifolds of K3$^{[n]}$-type, then there is a Hodge isometry $L(M) \simeq L(N)$ which restricts under the natural embeddings to the Hodge isometry on the sublattices $H^2(M,\bbZ) \simeq H^2(N,\bbZ)$ induced by the birational map between $M$ and $N$.
    \end{enumerate}
\end{theorem}
\begin{remark}
    Note that the embedding $H^2(M,\bbZ) \subset L(M)$ is only unique up to a lattice self-isometry of $L(M)$. Once such an embedding is fixed, the Hodge structure on $L(M)$ is determined by the one on $H^2(M,\bbZ)$ since $H^2(M,\bbZ)^\perp$ must be algebraic. 
\end{remark}
It is possible to generalize the moduli-theoretic interpretation of $L(M)$ to either the case where $M$ is a moduli space of Bridgeland stable objects on a (possibly twisted) K3 surface $S$ or in a Kuznetsov component $\Ku(X)$ of some cubic fourfold, at least for stability conditions in the distinguished component $\Stab^\dagger(S)$ (resp. $\Stab^\dagger(\Ku(X))$). The following result seems to be well-known to experts and has been suggested before in \cite[Cor.\ 1.3]{bayermacri} and in the comments following \cite[Thm.\ 29.2]{blmnps} but for the sake of completeness we make explicit the details of the deformation argument here.

\begin{proposition}\label{prop:markmanlatticegeneral}
    Suppose $M$ is one of the following cases: 
    \begin{enumerate}
        \item $M = M_\sigma(v)$ is a positive-dimensional moduli space of $\sigma$-stable objects on a twisted K3 surface $(S,\alpha)$, where $\sigma \in \Stab^\dagger(S,\alpha)$ is $v$-generic, or
        \item $M = M_\sigma(v)$ is a positive-dimensional moduli space of $\sigma$-stable objects on a Kuznetsov component $\cA_X$ for a cubic fourfold $X$, where $\sigma \in \Stab^\dagger(\cA_X)$ is $v$-generic.        
    \end{enumerate}
    In each case, assume that the Mukai vector $v$ is primitive and that $\sigma \in \Stab^\dagger$ lies in the distinguished component of the stability manifold. \par 
    Let $\cC$ denote the underlying K3 category in each of the cases (so $\cC$ is either $\Db(S,\alpha)$ or $\cA_X$). Then there is a Hodge isometry $L(M) \simeq \widetilde{H}(\cC,\bbZ) \coloneqq K^{\mathrm{top}}_0(\cC)$ which is compatible with the natural embeddings of $H^2(M,\bbZ).$
\end{proposition}

We will need the following lemma, which we state in terms of relative topological K-theory. 

\begin{lemma}\label{lem:twistedk3family}
    Let $(S,L)$ be a polarized K3 surface, $\alpha \in \Br(S)[n]$ and a Mukai vector $v \in \widetilde{H}^{1,1}(S,\alpha,\bbZ).$ Then there exists a smooth, irreducible affine curve $T$, a family of K3 surfaces $f : \cS \to T$, a relative polarization $\cL$ on $\cS$, a $\mu_n$-gerbe $\theta \in H^0(\cS,\mu_n)$ and a global section $\nu$ of the relative topological K-theory $K^{\mathrm{top}}((\cS,\theta)/T)$ such that
    \begin{enumerate}
        \item For all $t \in T$, $\nu_t \in K^{\mathrm{top}}((\cS,\theta)/T)$ is a Hodge class;
        \item $(S_0, L_0, \alpha_0, \nu_0) \simeq (S,  L, \alpha, v)$;
        \item and $\alpha_1 = 1 \in \Br(S_1)$,    
    \end{enumerate}
    where $\alpha_t \in \Br(\mathcal S_t)$ denotes the image of $\mathcal \theta_t$ via the natural map $H^2(\mathcal S_t, \mu_n) \to \Br(\mathcal S_t)$.
\end{lemma}
\begin{proof}
    Up to projecting the algebraic class $v \in \widetilde{H}(S,\alpha,\bbZ) \otimes \bbQ \simeq H^{\text{even}}(S,\bbQ)$ into the middle rational cohomology of $S$, scaling by an integer, and \cite[Lem.\ 18.5]{periodindex}, we may assume that $v \in \Pic(S)$. \par 
    Now let $\Lambda \subset \Pic(S)$ be the saturation of the rank 2 lattice spanned by $v$ and $L$, and let $T$ be a neighborhood of $S$ in the moduli space of $\Lambda$-polarized K3 surfaces. Let $0 \in U$ denote the point corresponding to $S$. Up to a finite étale base-change of $U$, there exists a universal K3 surface $f : \cS \to U$ as well as a family of $\mu_n$-gerbes in $\widetilde{\theta} \in H^0(\cS,R^2f_*\mu_n)$ lifting $\alpha$ at $0 \in U$. By the openness of ampleness, up to shrinking $U$ we may assume that the parallel transport of $L$ is ample everywhere. \par 
    Now if $S$ has Picard rank at least 3, \cite[Prop.\ 4.2.9]{breethesis} applies\footnote{Thanks to Dominique Mattei for pointing out the reference.} and there exist points $t \in U$ for which the image of $\widetilde\theta_t$ in $H^2(\cS_t,\bbG_m)$ vanishes; but by \cite[Rem.\ 4.2.10]{breethesis} since such points are dense in the moduli space of $\Lambda$-polarized K3 surfaces, the Picard rank assumption can be removed. Let $1 \in U$ be such a point where this property is satisfied, and take any smooth, irreducible affine curve $T$ joining $0$ and $1$ in $U$ and restrict $\cS$ to this curve $T$. Up to passing to a further cover of $T$, we may assume that $L$ is a genuine line bundle, and then by the Leray spectral sequence and Artin vanishing we may lift $\widetilde\theta$ to a global $\mu_n$-gerbe class $\theta \in H^2(\mathcal S, \mu_n).$ Then all the conditions required above are satisfied for the resulting family. 
\end{proof}

For the proof of the proposition we freely make use of the machinery of relative stability conditions established in \cite{blmnps}. 

\begin{proof}[Proof of Proposition \ref{prop:markmanlatticegeneral}]
    First observe that when $M$ is a K3$^{[n]}$-type hyperkähler manifold where $n \leq 3$, the result is automatic, because lattice-theoretically there is only one embedding of $H^2(M_\sigma(v),\bbZ)$ into the Mukai lattice (up to isometry of Mukai lattice) \cite[Thm.\ 9.1, Lem.\ 9.2 and 9.4]{markmanmukailattices}. So it suffices to consider the case when $n \geq 4$. \par
    In case (1) with $\alpha = 0$, we note by \cite[Thm.\ 1.1]{bayermacri} $M_\sigma(v)$ is birational to a Gieseker moduli space of stable sheaves, so it follows from Theorem \ref{thm:mmlattice}. So when $\alpha \neq 0$, or in case (2), the idea is to reduce to untwisted K3 surfaces by deformation. \par 
    Suppose we are in case (1) with $\alpha \in \Br(S)$ arbitrary. By Lemma \ref{lem:twistedk3family}, there exists a family of twisted polarized K3 surfaces $(\cS',\alpha')$ over $T$ which recovers $(S,\alpha)$ at $0 \in T$ and some untwisted K3 surface $S'$ at $1 \in T$. We first produce a relative stability condition over $T$ which coincides with $\sigma$ at $0 \in T$. By the (easier) analogue of \cite[Thm.\ 12.11]{periodindex} for surfaces, tilting of slope-stability produces a relative stability condition $\sigma'$ on $(\cS'/\alpha')$ over $T$ with respect to a lattice $\Lambda_{H}$ for $H$ the relative polarization. By the construction of $\Lambda_{H}$ in \cite[Ex.\ 11.3]{periodindex}, this lattice can be interpreted as the dual to the monodromy-invariant sublattice $M \subset K_{\mathrm{num}}(\cS,\alpha)$ spanned, up to B-field twist, by the powers of the polarization $H$. Using Yoshioka's trick \cite[Ex.\ 21.7 and Rem.\ 30.7]{blmnps}, we may enlarge the lattice $M$ at $0 \in T$ to $M = K_{\mathrm{num}}(S,\alpha)$. By definition, the stability condition $\sigma'_0$ produced by tilting lies in the connected component $\Stab^\dagger(S,\alpha)$, and is therefore deformation equivalent to $\sigma$. Then by the deformation trick of \cite[Rem.\ 30.6 and Prop.\ 30.8]{blmnps}, since the map $\Stab_{M^\vee}((\cS,\alpha')/T) \to \Hom(M^\vee,\bbC)$ is a connected covering space over an open subset containing the central charges $Z_{\sigma'}$ and $Z_{\sigma}$, we may lift the deformation in $\Stab^\dagger(S,\alpha)$ to a deformation in $\Stab_{M^\vee}((\cS',\alpha')/T)$ from $\sigma'$ to a relative stability condition which coincides with $\sigma$ at $0 \in T$. \par
    By \cite[Lem.\ 16.5]{periodindex}, up to a small deformation in $\Stab((\cS,\alpha)/T)$ we may also assume that $\sigma'_1$ is $v$-generic at $1 \in T$ without changing the moduli space $M = M_{\sigma'_0}(v)$ at $0 \in T$. Then by removing the finitely many remaining points of $T$ at which the stability condition is not $v$-generic, we may assume $\sigma'$ is $v$-generic at every point of $T$. Taking the relative moduli space $\pi : M_{\sigma'}(v) \to T$ gives a universal twisted sheaf over $\cS \times_T M'$ which is $(\alpha' \boxtimes \theta)$-twisted for some $\theta \in \Br(M')$. This gives a map of local systems 
    \[K^{\mathrm{top}}_0((\cS',-\alpha')/T) \to K^{\mathrm{top}}_0((M',\theta)/T) \xrightarrow{\mathrm{c}_2}R^4\pi_*(\underline{\bbZ}_{M'}) \to R^4\pi_*(\underline{\bbZ}_{M'})/\mathrm{Sym}^2R^2\pi_*(\underline{\bbZ}_{M'}),\]
    where $c_2$ is the second Chern class. Note that in the quotient $R^4\pi_*(\underline{\bbZ}_{M'})/\mathrm{Sym}^2R^2\pi_*(\underline{\bbZ}_{M'})$, the second Chern class $c_2$ coincides with $\mathrm{ch}_2 = \mathrm{ch}_2^B$ for any $B$-field lift of $\theta$, so in particular fiber-by-fiber this map of local systems is also a map of Hodge structures. But by \cite[Thm.\ 1.14]{markmanmonodromy} and the characterization in part (c) of Theorem \ref{thm:mmlattice}, at $1 \in T$ this is a Hodge isometry $\widetilde{H}(S'_1,\bbZ) \to L(M_{\sigma'_1}(v))$ which up to scaling sends $v_1^\vee$ to $c_2(T_{M_{\sigma}'(v)})$. By parallel transport it must also be an isomorphism of Hodge structures at $0 \in T$, and the monodromy-invariance of the bilinear pairing implies that it is also an isometry at $0 \in T$. Hence $\widetilde{H}(S,-\alpha,\bbZ)$ is Hodge-isometric to $L(M_{\sigma}(v))$. Composing with the Hodge isometry $\widetilde{H}(S,\alpha,\bbZ) \simeq \widetilde{H}(S,-\alpha,\bbZ)$ coming from the anti-autoequivalence $(-)^\vee$ of $\Db(S,\alpha)$ and $\Db(S,-\alpha)$ completes the claim in case (1). \par 
    Case (2) can be proved by a similar deformation argument. By \cite[Cor.\ 32.1 and Prop.\ 32.4]{blmnps}, there exists a family of cubic fourfolds (and hence Kuznetsov components) giving a deformation of the moduli space $M_\sigma(v)$ to a moduli space of stable objects on a twisted K3 surface $(S,\alpha)$, so by repeating the argument we reduce to the previous case, and the proposition follows. 
\end{proof}
By specializing to the large volume limit for $\Stab^\dagger(S,\alpha),$ the result holds for moduli spaces of (twisted) stable sheaves as well.
As a consequence of this and the twisted derived Torelli theorem, we find: 
\begin{corollary}\label{cor:birimpliesequiv}
    Let $M_1$ and $M_2$ be positive-dimensional birational moduli spaces of Bridgeland-stable objects on K3 categories $\cC_1$ and $\cC_2$, where each $\cC_i$ is either the derived category of a twisted K3 surface or the Kuznetsov component of a cubic fourfold, with stability conditions $\sigma_i \in \Stab^\dagger(\cC_i).$ Then there is a Hodge isometry
    \[\widetilde{H}(\cC_1,\bbZ) \simeq \widetilde{H}(\cC_2,\bbZ).\]
    Moreover, if both $\cC_i$ are equivalent to twisted derived categories $\Db(S_i,\alpha_i)$ and either $\alpha_1 = 0$, $\cC_1 \simeq \Ku(X)$ for some smooth cubic fourfold $X$, or $S_1$ has Picard rank $\geq 12$, then in fact we have an equivalence
    \[\cC_1 \simeq \cC_2.\]
\end{corollary}
\begin{proof}
    Combining Theorem \ref{thm:mmlattice}(b) and (d) with the previous proposition immediately yields the first result. \par 
    Assume now that $\cC_i \simeq \Db(S_i,\alpha_i)$, so we have an isometry 
    \[\widetilde{H}(S_1,\alpha_1,\bbZ) \simeq \widetilde{H}(S_2,\alpha_2,\bbZ).\]
    If this isometry is orientation-preserving, then the twisted derived Torelli theorem \cite[Thm.\ 0.1]{twistedtorelli} applies and $\cC_1 \simeq \cC_2$. If it is not orientation-preserving, either condition that $\alpha_1 = 0$ or $\cC_1 \simeq \cA_X$ implies $\widetilde{H}(S_1,\alpha_1,\bbZ)$ admits an orientation-reversing Hodge self-isometry \cite[Lem.\ 2.3 and Rem.\ 2.4]{k3category}; similarly any twisted K3 with Picard rank $\geq 12$ admits such an orientation-reversing isometry \cite[Cor.\ 2.8]{reinecke}. It follows that up to composing with this self-isometry we may again apply the twisted derived Torelli theorem. 
\end{proof}

\section{Twisted derived categories of moduli spaces on Kuznetsov components}
\label{sec:kuznetsovcomponent}
By applying the results above, we are able to generalize the main theorem of \cite{huybrechtsbottini} to moduli spaces on Kuznetsov components of cubics besides the Fano variety of lines.

\begin{theorem}\label{thm:main_kuznetsov_component}
    Let $Y$ be a smooth cubic fourfold. Let $v \in \widetilde{H}(\cA_Y, \mathbb Z)$ be a non-zero primitive vector and let $\sigma$ be a $v$-generic stability condition in $\Stab^{\dagger}(\mathcal A_Y)$. Furthermore, assume that the moduli space
    $M_{\sigma}(\mathcal A_Y, v)$
    is of Picard rank $2$ and admits a rational Lagrangian fibration such that the corresponding isotropic class in $H^2(X, \mathbb Z)$ is indivisible. Then, there is a Brauer class $\theta \in \Br(M_{\sigma}(\mathcal A_Y, v))$ and an exact linear equivalence
    $$D^b(M_{\sigma}(\mathcal A_Y, v), \theta^n) \simeq \mathcal A_Y^{[n]},$$
    where $2n = v^2+2$ is the dimension of $M_{\sigma}(\mathcal A_Y, v)$.
\end{theorem}
\begin{remark}
As mentioned before, the square of the divisibility of the isotropic class defining the Lagrangian fibration always divides $n-1$, see \cite[Lem.\ 2.5]{markmanlagrangian}. Hence, the condition on the divisibility is automatically satisfied if $n-1 = v^2/2$ is square-free.
\end{remark}
\begin{proof}
    Since $M_{\sigma}(\cA_Y, v)$ is a non-special rationally Lagrangian fibered hyperkähler manifold of K3$^{[n]}$-type, there is a polarized K3 surface $(S, L)$ and a special Brauer class $\alpha \in \SBr(S, L)$ such that $M_{\sigma}(\cA_Y, v)$ is birational to $\Pic^0_{\alpha}(S, L)$ by \cite[Thm.\ 1.2]{huybrechtsmattei}.
    Since $\rho(S) = \rho(M_{\sigma}(\cA_Y, v)) -1 = 1$, the assumption on the divisibility of the pullback of $\mathcal O_{\mathbb P^n}(1)$ ensures that $\Pic(S) = \mathbb Z L$ by \cite[Thm.\ 1.5]{markmanlagrangian} and, in particular, every member of the linear system $|L|$ is integral. Thus, the compactified twisted Jacobian $\overline\Pic^0_{\alpha} \coloneqq \overline{\Pic}^0_\alpha(S, L)$ is a smooth hyperkähler manifold of K3$^{[n]}$-type, which is birational to the moduli space $M_\sigma(\cA_Y, v)$.

    Let us now show that there is an exact linear equivalence $\cA_Y \simeq D^b(S, \overline{\alpha})$, where $\overline{\alpha}$ denotes the image of $\alpha$ in $\Br(S)$. Since $M_\sigma(\cA_Y, v)$ and $\overline{\Pic}^0_\alpha$ are birational moduli spaces of stable objects in $\Db(S,\alpha)$ and $\cA_Y$ respectively, we will have
    $$\cA_Y \simeq D^b(S, \overline{\alpha})$$
    by Corollary \ref{cor:birimpliesequiv} as long as we know that $\cA_Y$ is equivalent to a possibly different twisted K3 surface $\Db(S',\alpha')$. On the other hand, the existence of a rational Lagrangian fibration on $M_\sigma(\cA_Y,v)$ implies the existence of an isotropic algebraic class with respect to the BBF form in 
    \[H^2(M_\sigma(\cA_Y,v),\bbZ) \subset \widetilde{H}(\cA_Y,\bbZ).\]
    In particular it follows by \cite[Prop.\ 33.1]{blmnps} that there exists a twisted K3 surface such that $\cA_Y \simeq \Db(S',\alpha')$, so Corollary \ref{cor:birimpliesequiv} applies. 
    
    By Corollary \ref{cor:twistedpicequiv}, there is an exact linear equivalence
    $$D^b(\overline{\Pic}^0_{\alpha},o_\alpha([\overline\Pic^{n}])^{-1}) \simeq D^b(\overline{\Pic}^{n}, o_{n}(\alpha)) \simeq D^b(S^{[n]}, \overline{\alpha}^{[n]}).$$
    By twisted BKR, cf.\ \cite[Prop.\ 1.2]{huybrechtsbottini}, there is an equivalence
    $$D^b(S^{[n]}, \overline{\alpha}^{[n]}) \simeq (D^b(S, \overline{\alpha}))^{[n]}.$$
    Fix a birational map between $\overline{\Pic}^0_{\alpha}$ and $M_{\sigma}(\cA_Y, v)$. The induced isomorphism on Brauer groups identifies
    $o_\alpha([\overline\Pic^{n}])^{-1} \in \Br(\overline{\Pic}^0_\alpha)$ with a class $\theta \in \Br(M_\sigma(\cA_Y, v))$.
    The claim then follows by composing the two equivalences above with the equivalence
    $$D^b(\overline{\Pic}^0_\alpha,o_\alpha([\overline\Pic^{n}])^{-1}) \simeq D^b(M_\sigma(\cA_Y, v), \theta)$$
    induced by the birational map between $\overline{\Pic}^0_{\alpha}$ and $M_{\sigma}(\cA_Y, v)$, see \cite[Thm.\ 0.3]{dequivalence}.
\end{proof}
\begin{remark}
    Suppose that $v = a\lambda_1 + b\lambda_2$, where $\lambda_i = v(\mathrm{pr}_{\cA_Y}(\cO_\ell(i)))$. Then the condition that the moduli space $M_\sigma(\cA_Y,v)$ is of Picard rank 2 is nonempty (and therefore very general) among those cubic fourfolds such that $M_\sigma(\cA_Y,v)$ admits a rational Lagrangian fibration. \par Indeed, the same holds for polarized hyperkähler varieties of K3$^{[n]}$-type and by \cite[Cor.\ 29.5]{blmnps} moduli spaces of objects with this Mukai vector in Kuznetsov components form a locally complete family of polarized hyperkähler manifolds of K3$^{[n]}$-type.
\end{remark}

For any cubic fourfold $Y$, Lehn, Lehn, Sorger and van Straten have constructed a certain eight-dimensional hyperkähler manifold $Z(Y)$ of K3$^{[4]}$-type arising as a particular contraction of the moduli space of twisted cubics \cite{llsvs}. For very general $Y$ such that $Z(Y)$ admits a rational Lagrangian fibration, we are able to apply the theorem above.

\begin{corollary}\label{cor:llsvs}
    Let $Y$ be a cubic fourfold with $\operatorname{rk} H^{2, 2}(Y, \mathbb Z) = 2$ for which $Z(Y)$ admits a rational Lagrangian fibration. Then there is a Brauer class $\theta \in \Br(Z(Y))$ and an exact linear equivalence
    $$D^b(Z(Y), \theta^4) \simeq \mathcal A_Y^{[4]}.$$
\end{corollary}

\begin{proof}
    By \cite{lipertusizhao}, we have
    $$Z(Y)  \simeq M_{\sigma}(\mathcal A_Y, v),$$
    for $v = \lambda_1+2\lambda_2$ with $v^2=6$. The claim then follows from  Theorem \ref{thm:main_kuznetsov_component}.
\end{proof}

\begin{remark}\label{rem:genericity_2}
By the same argument as in Remark \ref{rem:genericity}, the assumption on the Picard rank in Theorem \ref{thm:main_kuznetsov_component} and Corollary \ref{cor:llsvs} can be weakened to a genericity assumption. In particular, the conclusion of Corollary \ref{cor:llsvs} holds on a Zariski open subset of the Hassett divisors on which $Z(Y)$ admits a Lagrangian fibration.
\end{remark}
\printbibliography
\end{document}